\documentclass[11pt]{amsart}
\usepackage{amssymb, amsmath, enumerate, verbatim, fullpage,url}
\usepackage[all]{xy}

\newcommand{\nnd}{^{\operatorname{nd}}}
\newcommand{\tth}{^{\operatorname{th}}}
\newcommand{\rrd}{^{\operatorname{rd}}}
\newcommand{\ZZ}{\mathbb{Z}}
\newcommand{\QQ}{\mathbb{Q}}
\newcommand{\PP}{\mathbb{P}}

\newcommand{\Aff}{\mathbb{A}}

\newcommand{\OO}{\mathcal{O}}

\newcommand{\pc}[1]{Y^{\operatorname{pre}}\left(#1\right)}
\newcommand{\cpc}[1]{X^{\operatorname{pre}}\left(#1\right)}
\renewcommand{\bar}{\overline}

\newcommand{\col}{\,{:}\,}

\DeclareMathOperator{\codim}{codim}
\DeclareMathOperator{\tors}{tors}

\DeclareMathOperator{\Spec}{Spec}

\newtheorem{thm}{Theorem}[section]
\newtheorem*{thm*}{Theorem}

\newtheorem{cor}[thm]{Corollary}

\newtheorem*{conjecture*}{Conjecture}

\theoremstyle{remark}

\newtheorem{rem}[thm]{Remark}

\theoremstyle{definition}
\newtheorem{defn}[thm]{Definition}

\begin{document}

\title{Pre-images in Quadratic Dynamical Systems}

\author[Hutz]{Benjamin Hutz}
\address{
The Graduate Center\\
The City University of New York \\
New York, NY 10016 \\ USA
}
\email{bhutz@gc.cuny.edu}

\author[Hyde]{Trevor Hyde}
\address{
Department of Mathematics and Computer Science \\
Amherst College \\
Amherst, MA 01002 \\ USA
}
\email{thyde12@amherst.edu}

\author[Krause]{Benjamin Krause}
\address{
Department of Mathematics and Computer Science \\
Amherst College \\
Amherst, MA 01002 \\ USA
}
\email{bkrause10@amherst.edu}

\begin{abstract}
	For a quadratic polynomial with rational coefficients, we consider the problem of bounding
    the number of rational points that eventually land at a given constant after iteration, called pre-images of the constant.  In the article ``Uniform Bounds on Pre-Images Under Quadratic Dynamical Systems,'' it was shown that the number of rational pre-images is bounded as one varies the polynomial. Explicit bounds on the number of pre-images of zero and $-1$ were addressed in subsequent articles.   This article addresses explicit bounds on the number of pre-images of any algebraic number for quadratic dynamical systems and provides insight into the geometric surfaces parameterizing such pre-images.
\end{abstract}

\thanks{The authors wish to thank David Cox for patiently answering many geometric questions, the Amherst Computing Cluster for providing CPU time, and the many helpful comments and suggestions from the referee.}



\date{}

\maketitle


\section{Introduction}

	Fix an algebraic number field $K$ and a number $c \in K$ and define an endomorphism of the affine line by
	\begin{equation*}
		f_c: \Aff^1_K \to \Aff^1_K, \qquad f_c(x) = x^2 + c.
	\end{equation*}
    If we define $f_c^N$ to be the $N$-fold composition of the morphism $f_c$, and $f_c^{-N}$ to be the inverse image of $a$ in $\Aff^1_K$ under $f_c^{N}$, then for $a \in \Aff^1(K)$, the set of \textbf{rational iterated pre-images of $a$} is given by
	\begin{equation*} \label{Eqn: Preimages}
		\bigcup_{N \geq 1} f_c^{-N}(a)(K)
			= \{x_0 \in \Aff^1(K): f_c^N(x_0) = a \text{ for some $N \geq 1$}\}.
	\end{equation*}

    Heuristically, finding iterated pre-images amounts to solving progressively more complicated polynomial equations, and so $K$-rational solutions should be a rarity.  The situation becomes more interesting as we vary $c$, which has the effect of varying the morphism $f_c$.
    \begin{defn}
        Define
        \begin{equation*}
        		\kappa(a) = \sup_{c \in K}
        			\# \left\{\bigcup_{N \geq 1} f_c^{-N}(a)(K) \right\}.
        	\end{equation*}
    \end{defn}
    A special case of the main theorem in \cite{FHIJMTZ} shows that $\kappa(a)$ is finite, but does not give an explicit bound. Note that it is easy to construct a pair $(a,c)$ with arbitrarily many rational pre-images simply by fixing $c$ and taking $a=f_c(N)(0)$.  The fact that $\kappa(a)$ is finite shows that, for a given $a$, such $c$ values are rarely defined over the same field.

    When needed for clarity, we include the field $K$ in the notation as $\kappa(a,K)$.  In this article, we focus on a weaker notion $\bar{\kappa}(a)$ that bounds the ``typical'' number of rational pre-images.
    \begin{defn}
        Define
        \begin{equation*}
            \bar{\kappa}(a,K) = \limsup_{c \in K} \# \left\{\bigcup_{N \geq 1} f_c^{-N}(a)(K) \right\}.
        \end{equation*}
    \end{defn}
    In essence $\bar{\kappa}(a)$ differs from $\kappa(a)$ by excluding at most finitely many $c$ values from consideration, thus, $\bar{\kappa}(a) \leq \kappa(a)$.

    The cases of $a=0$ and $a=-1$ were studied in \cite{FH,Hyde}, respectively, and it was shown that $\bar{\kappa}(0,\QQ)=\bar{\kappa}(-1,\QQ)=6$.  In \cite{FH} a significant amount of effort went into the more difficult task of showing that $\kappa(0,\QQ) = 6$, assuming some standard conjectures.  This article addresses the situation from the more general setting of allowing $a$ to vary and examining the ``pre-image surfaces'' instead of ``pre-image curves.''  We also allow arbitrary number fields $K$.  Our main result is the following theorem.
    \begin{thm} \label{thm_kappabar}
        For $a \in \bar{\QQ}$ and for any fixed algebraic number field $K$ we have
        \begin{equation*}
            \bar{\kappa}(a,K) = \begin{cases}
              10 & a=-\frac{1}{4}\\
              6 \text{ or } 8  & a = \text{ the three $3\rrd$ critical values} \\
              4 & a \in S \cap K\\
              6 & \text{otherwise}.
            \end{cases}
        \end{equation*}
        The set $S$ is the finite set of $a$ values (in $\bar{\QQ}$) where the elliptic surface with two rational first pre-images and four rational second pre-images and the elliptic surface with two rational first pre-images, (at least) two rational second pre-images, and (at least) two rational third pre-images both have specialization with rank zero at $a$.
    \end{thm}
    The elliptic surface parameterizing values of $a$ and $c$ with two rational first pre-images, (at least) two rational second pre-images, and (at least) two rational third pre-images has generic rank two (Theorem \ref{thm222_rank}). Thus, finding the set of $a$ values where the corresponding specialization is an elliptic curve of rank zero is a generalization of the problem studied in \cite{Masser}.  Masser and Zannier have shown that such sets are finite \cite{Zannier}, implying the set $S$ is finite.  The critical values are defined in Definition \ref{defn_critical_values}.

    The organization of the article is as follows.  In Section \ref{sect_lower} we examine the lower bound for $\bar{\kappa}(a)$ by finding the generic rank over $\QQ$ of the elliptic surfaces corresponding to arrangements of $6$ pre-images.  In Section \ref{sect_upper} we examine the upper bound on $\bar{\kappa}(a)$ by showing that all arrangements of $2N$ pre-images for some $N$ correspond to curves of genus greater than $1$.  In Section \ref{sect_proof} we prove Theorem \ref{thm_kappabar}.  In Section \ref{sect_other} we prove some additional properties of the pre-image surfaces that are tangential to the proof of Theorem \ref{thm_kappabar}, yet still of interest.  Section \ref{sect_torsion} parameterizes the possible torsion subgroups of the elliptic surface corresponding to two rational first pre-images and four rational second pre-images.
    Section \ref{sect_exceptional} examines exceptional pairs $(a,c)$ that are excluded by considering $\bar{\kappa}(a)$ instead of $\kappa(a)$.

	We present these results for two reasons. First, by working with the ``moduli surfaces'' parameterizing arrangements of pre-images, our
    problem can be reduced to the classical Diophantine problem of finding rational points on curves and surfaces.  Second, our setting provides a nice example in which elliptic surfaces naturally arise and we apply specialization theorems, rank arguments, height functions, and use explicitly that the geometry of a curve has implications for its arithmetic through the use of Falting's Theorem.

    We make heavy use of the algebra and number theory systems \textit{Magma} \cite{magma} and \textit{PARI/gp} \cite{pari232}.
    	
    A similar analysis would almost certainly be possible for the families of maps of the form $x^d +c$ for $d\geq 2$ a positive integer.  In fact, for any family of polynomial maps of fixed degree it seems likely that the same methods would apply.  For more general rational maps, at the very least, there would be additional complications for the genus calculations.  This problem poses an interesting direction for further study.

\section{Pre-image Curves and Surfaces}
\label{Sec: Geometry}

    In this section we summarize the necessary geometric theory of pre-image curves developed in \cite{FHIJMTZ,FH} and then introduce the pre-image surfaces we consider in this article.  Let $K$ be a number field.  As in the introduction, we define a morphism $f_c: \Aff^1_K \to \Aff^1_K$ for any $c \in K$ by the formula $f_c(x) = x^2 + c$.
    We could view $f_c$ as an endomorphism of $\PP^1_K$, but the point at infinity is totally invariant for this type of morphism and, thus, dynamically uninteresting. Fix a point $a \in K$ and a positive integer $N$. Define an algebraic set
    \begin{equation*}
    	\pc{N,a} = V\left(f_c^N(x) - a\right) \subset \Aff^2_K = \Spec K[x,c].
    \end{equation*}
    If $\pc{N,a}$ is geometrically irreducible, we define the \textbf{$N\tth$ pre-image curve}, denoted $\cpc{N,a}$, to be the unique complete curve birational to $\pc{N,a}$.
    \begin{defn} \label{defn_critical_values}
        We say $a$ is an \emph{$N\tth$ critical value} of $f_c$ if
        \begin{equation*}
            f_{c_0}^N(0) = a \quad \text{and} \quad \frac{df_c^N(0)}{dc}\Big|_{c=c_0} = 0.
        \end{equation*}
    \end{defn}
    \begin{thm}[{\cite[Cor.~2.4 \& Thm.~3.2]{FHIJMTZ}}] \label{thm_genus_formula}
        Suppose $N$ is a positive integer and $a \in K$ is not a critical value of
        $f_c^j$ for any $2\le j\le N$.  Then $\pc{N,a}$ is nonsingular, geometrically irreducible, and the genus of $\cpc{N,a}$ is $(N-3)2^{N-2} + 1$.
    \end{thm}
    For $a\in K$, define a morphism $\psi\colon \pc{N,a} \to \Aff^N$ by
    \begin{equation*}
        \psi(x,c) = \left(x, f_c(x), f_c^2(x), f_c^3(x), \ldots, f_c^{N-1}(x) \right).
    \end{equation*}
    We recall the following theorem.
    \begin{thm}[{\cite[Proposition 4.2]{FH}}] \label{thm_FH_model}
        \mbox{}
       	\begin{enumerate}
    		\item\label{compint} The projective closure of the image of $\psi$ is a complete intersection of
                quadrics with homogenous ideal
				\begin{equation*}
					J = \left(Z_{N-1}^2 + Z_iZ_N - Z_{i-1}^2 - aZ_N^2 : i = 1, 2, 3, \ldots, N-1 \right).
				\end{equation*}
    			    		
    		\item\label{kvalpt} The points of $V(J)$ on the hyperplane $Z_N = 0$ have homogeneous coordinates
    		    \begin{equation*}
    				\left( \epsilon_0 : \cdots : \epsilon_{N-1} : 0 \right), \qquad \epsilon_i = \pm 1.
                \end{equation*}
                In particular, there are $2^{N-1}$ of them. Moreover, they are all nonsingular points of $V(J)$.
    			
    		\item\label{2^N-1pts} If $\pc{N,a}$ is nonsingular, then $\cpc{N,a} \cong V(J)$ and the complement
                of the affine part $\cpc{N,a} \smallsetminus  \pc{N,a}$	consists of $2^{N-1}$ points.
    	\end{enumerate}
    \end{thm}

    \begin{defn}
        We define the \emph{$N\tth$ pre-image surface} $\cpc{N}$ as the surface fibered over $\PP^1_K$ by $a$.  The fiber over $a$ is given by $\cpc{N,a}$ if $\pc{N,a}$ is geometrically irreducible and $V(J)$ otherwise.  In particular, for each $a \in K$ not a critical value of $f_c$, we get a nonsingular curve in $\PP^N_K$.
        \begin{equation*}
            \xymatrix{\cpc{N} \ar[d]^{\pi} & \cpc{N,a} \ar@{|->}[d]^{\pi}\\ \PP^1_K & a}
        \end{equation*}
        Note that for a fixed $a_0$, the affine points $(x_0,c_0,1)$ on the curve $\cpc{N,a_0}$ are in bijection with the $N\tth$ pre-images $x_0 \in f_{c_0}^{-N}(a_0)$.
    \end{defn}
    We will consider the $N\tth$ pre-image surfaces in the language of function fields.  In particular, consider the function field $K(a)$ which is comprised of all rational functions in $a$ with $K$-rational coefficients.  In particular, we consider the surfaces defined as
    \begin{equation*}
        \pc{N} = V\left(f_c^N(x) - a\right) \subset \Aff^2_{K(a)}
    \end{equation*}
    and
    \begin{equation*}
        \cpc{N} = V\left(Z_{N-1}^2 + Z_iZ_N - Z_{i-1}^2 - aZ_N^2 : i = 1, 2, 3, \ldots, N-1 \right) \subset \PP^N_{K(a)}.
    \end{equation*}
    The genus formula (Theorem \ref{thm_genus_formula}) applies to each fiber for which $\pc{N,a}$ is nonsingular and geometrically irreducible.  In particular, $\cpc{1}$ and $\cpc{2}$ have fibers of genus $0$, $\cpc{3}$ has fibers of genus $1$, and $\cpc{N}$ for $N \geq 4$ has fibers of genus $> 1$ (with finitely many exceptional fibers for each $N$).  Therefore, for $N>3$ and all but finitely many $a \in K$, it follows from Falting's theorem that there are only finitely many points $(x,c) \in \cpc{N,a}$.  Thus, except for the finitely many $a$ values, the $N\tth$ pre-images for $N > 3$ have no contribution to $\bar{\kappa}(a)$.  This premise is the content of Corollary \ref{core_nonsing_a} and the rest of Section \ref{sect_upper} addresses the exceptional $a$ values.

    Throughout this article we discuss arrangements of pre-images.  For example, by a 222 arrangement we mean that there are two rational first pre-images, (at least) two rational second pre-images, and (at least) two rational third pre-images. Similarly, a 2424 arrangement has two rational first pre-images, four rational second pre-images, (at least) 2 rational third pre-images, and (at least) four rational fourth pre-images.  Note that any 226 arrangement would have to be part of a 246 arrangement since the forward image of a rational point is still a rational point.

\section{Arrangements of Six Pre-Images} \label{sect_lower}
    By examining the arrangements of six pre-images we are able to prove the following lower bound for $\bar{\kappa}(a)$.
    \begin{thm}\label{thm_lower}
        Let $K$ be a number field.  There is a finite set $S$ such that
        \begin{equation*}
            \begin{cases}
              \bar{\kappa}(a) \geq 6 & a \in K \backslash (S \cap K)\\
              \bar{\kappa}(a) = 4 & a \in S\cap K.
            \end{cases}
        \end{equation*}
    \end{thm}
    \begin{proof}
        The 22 curve over the function field $K(a)$ is the curve whose points correspond to two rational first pre-images and (at least) two rational second pre-images.  It has fibers of genus $0$ \cite{FHIJMTZ} and at least one $\QQ$-rational section for each choice of $a$, $(1,1,0)$. Thus, each fiber has infinitely many rational points and $\bar{\kappa}(a) \geq 4$.

        Theorem \ref{thm222_rank} shows that the $222$ surface has generic rank at least $2$ (exactly $2$ over $\QQ$).  Theorem \ref{thm24_rank} shows that the $24$ surface has generic rank $0$ over $\QQ$.  Let $S$ be the (possibly empty) set of $a$ values for which both the $222$ and $24$ surface specialize to rank $0$.  By \cite{Zannier} the set of $a$ values where the $222$ surface has rank $0$ is finite and thus, $S$ is finite.  If $a \in S \cap K$, $\bar{\kappa}(a) = 4$, otherwise $\bar{\kappa}(a) \geq 6$.
    \end{proof}

\subsection{Second Pre-Images}
    We consider the situation where the pre-image tree is full to the second level.  In other words, two rational first pre-images and four rational second pre-images,
    
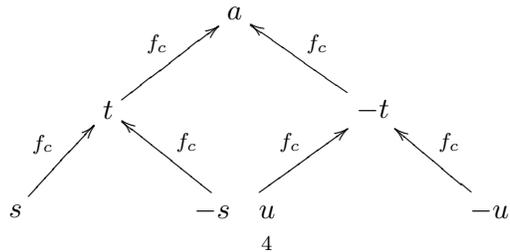
\begin{equation*}
        \xymatrix{& & a  & &\\
            & t \ar[ur]^{f_c}&  & -t\ar[ul]_{f_c}  &\\
            s \ar[ur]^{f_c} & & -s \quad u \ar[ul]_{f_c} \ar[ur]^{f_c}& &-u. \ar[ul]_{f_c}}.
    \end{equation*}
    We can define this curve over the function field $K(a)$ as
    \begin{equation*}
        X_{24} = V(s^2 - tz - (t^2-az^2), u^2+tz - (t^2-az^2)) \subseteq \PP^3_{K(a)}.
    \end{equation*}
    The fibers (when nonsingular) have genus one with at least one rational section $(1,1,1,0)$ so we can produce a minimal Weierstrass model (using Magma \cite{magma}) as an elliptic curve over the function field $K(a)$ as
    \begin{equation*}
        E_{24}(a): v^2w = u^3 + (4a - 1)u^2w + 16auw^2 + (64a^2 - 16a)w^3
    \end{equation*}
    with $j$-invariant
    \begin{equation*}
        j(a) = \frac{(16a^2 - 56a + 1)^3}{a(4a+1)^4}
    \end{equation*}
    and discriminant
    \begin{equation*}
        \Delta(a) = a(4a+1)^4.
    \end{equation*}
    The only fibers which are not elliptic curves are $a=0$ and $a=-\frac{1}{4}$.  This is in fact a rational elliptic surface since it has a Weierstrass model satisfying $\deg(a_i) \leq i$ for $a_i$ the coefficients of an elliptic curve in Weierstrass form \cite[page 237]{shioda}.

    \begin{thm} \label{thm24_rank}
        $E_{24}(a)(\QQ(a))$ has rank $0$ and torsion subgroup $\ZZ/4\ZZ$ generated by
        \begin{equation*}
            T(a) = (2,8a+2,1).
        \end{equation*}
    \end{thm}
    \begin{proof}
        We use \cite[Main Theorem]{oguiso} to see that the rank over $\QQ(a)$ is zero.  We compute the Kodaira symbols in Magma to get
        \begin{equation*}
            [ <I4, 1>, <I1*, 1>, <I1, 1> ].
        \end{equation*}
        From row 72 in the table \cite{oguiso} we have that the rank of $E_{24}(a)(\QQ(a))$ is zero.  Examining the torsion, we see that the point
        \begin{equation*}
            (2,8a+2,1)
        \end{equation*}
        has order $4$ and the specialization $E_{24}(1)(\QQ)$ has torsion subgroup $\ZZ/4\ZZ$.  Since the specialization map is injective on torsion on all nonsingular fibers, $E_{24}(a)$ has torsion subgroup exactly $\ZZ/4\ZZ$.
    \end{proof}

\subsection{Third Pre-Images}
    From Theorem \ref{thm_FH_model} we see that the elliptic surface parameterizing third pre-images of $a$ over the function field $K(a)$ is given by
    \begin{equation*}
        X_{222}=V(z_2^2 + z_1z_3 - z_0^2 - az_3^2, z_2^2 + z_2z_3 - z_1^2 - az_3^2) \subseteq \PP^{3}_{K(a)}.
    \end{equation*}
    Using the cuspidal point $(-1,1,1,0)$ from Theorem \ref{thm_FH_model} as the section at infinity we can find a minimal model in Magma as
    \begin{align*}
        E_{222}(a):v^2w &= u^3 + (16a + 942/13)u^2w + (10048/13a + 293084/169)uw^2 + (1024a^2 \\
            &+ 1620800/169a + 30250696/2197)w^3
    \end{align*}
    with $j$-invariant
    \begin{equation*}
        j(a) =\frac{(16a^2 + 3)^2}{(4a + 1)^2(256a^3 + 368a^2 + 104a + 23)}
    \end{equation*}
    and discriminant
    \begin{equation*}
        \Delta(a) = (4a + 1)^2(256a^3 + 368a^2 + 104a + 23).
    \end{equation*}
    As expected, the only fibers which are not elliptic curves are the fibers over $a= -1/4$ and the three $3\rrd$ critical values.
    This is in fact a rational elliptic surface since it has a Weierstrass model satisfying $\deg(a_i) \leq i$ for $a_i$ the coefficients of an elliptic curve in Weierstrass form \cite[page 237]{shioda}.

    \begin{thm} \label{thm222_rank}
        $E_{222}(a)(\QQ(a))$ has rank $2$ generated by the two independent sections
        \begin{equation*}
            P(a) = \left(-\frac{262}{13},32a+8,1\right) \quad \text{and} \quad Q(a) = \left(-\frac{366}{13},32a+8,1\right).
        \end{equation*}
    \end{thm}
    \begin{proof}
        We use \cite[Main Theorem]{oguiso} to see that the rank over $\QQ(a)$ is exactly two.  We compute the Kodaira symbols in Magma to get
        \begin{equation*}
            [ <I1, 3>, <I2, 1>, <I1*, 1> ].
        \end{equation*}
        From row 30 in the table \cite{oguiso} we have that the rank of $E_{222}(a)(\QQ(a)) = 2$.  Since the specialization map is injective on torsion on all fibers where $E_{222}$ is nonsingular, and the specialization $E_{222}(0)$ has no torsion, there are no rational torsion sections.
        We can see $P(a)$ and $Q(a)$ are actually the generators by finding a specialization $E_{222}(a_0)$ which is rank $2$ with generators $P(a_0)$ and $Q(a_0)$.  For $a=4$ we have
        \begin{align*}
            E_{222}(4): v^2w = u^3 + 1774/13u^2w + 815580/169uw^2 + 150527944/2197w^3
        \end{align*}
        and from Magma the generators are
        \begin{equation*}
            (-262/13, 136, 1)\quad \text{and} \quad (-1146/13, 136, 1).
        \end{equation*}
        In terms of $P(4)$ and $Q(4)$ these are
        \begin{equation*}
            P(4) \quad \text{and} \quad P(4) + Q(4).
        \end{equation*}
        Thus, $P(4)$ and $Q(4)$ generate the Mordell-Weil group $E_{222}(4)$ and, hence, $P(a)$ and $Q(a)$ generate the Mordell-Weil group of $E_{222}(a)$.
    \end{proof}


\section{Arrangements of Eight or More Pre-Images} \label{sect_upper}
   We examine when the genus of the fibers of pre-image surfaces of various arrangements of $2N$ pre-images is greater than $1$ and, thus, by Falting's theorem have a finite number of rational points over an algebraic number field.  In particular, if every $2N$ arrangement has genus greater than $1$ for some $N$, then $\bar{\kappa}(a) < 2N$.  The difficulty lies in determining the genus when the fiber is singular.  We treat the nonsingular case in the following theorem.
    \begin{thm}
        If the curve (fiber) defining an arrangement of $2N$ rational pre-images of $a$ is nonsingular, then it has genus $(N-3)2^{N-2}+1$.
    \end{thm}
    \begin{proof}
        A complete intersection in $\PP^m$ is defined as a subscheme $Y$ of $\PP^m$ whose homogeneous ideal $I$ can be generated by $r = \codim(Y,\PP^m)$ elements \cite[Exercise II.8.4]{Hartshorne}.  Each surface arranging $2N$ points can be described by the equations
        \begin{equation*}
            f_c(z_1) = a \quad \text{ and } \quad f_c(z_i) = (-1)^{\epsilon} z_{j} \text{ for } 2 \leq i \leq N
        \end{equation*}
        where $1 \leq j < N$ and $\epsilon = \pm 1$ depending on the arrangment of points.  Homogenizing and eliminating $c$ from this system of equations describes each fiber as a curve defined by $N-1$ degree two hypersurfaces in $\PP^{N}$ and, hence, a complete intersection.  From \cite[\S 22]{Hirzebruch} or \cite[Corollary 2]{Arslan} we get a formula for the arithmetic genus of a complete intersection of $N-1$ degree two hypersurfaces in $\PP^{N}$ as
        \begin{equation*}
            p_a = \sum_{m=1}^{N-1} (-1)^{m+1} \binom{N-1}{m} \phi_N(-2m)
        \end{equation*}
        where $\phi_N(z)$ comes from the Hilbert polynomial of the $2N$ curve and is given by
        \begin{equation*}
            \phi_N(z) = \frac{(z+1)(z+2) \cdots (z+N)}{N!} = \binom{z+N}{N}.
        \end{equation*}
        Since the arithmetic genus is equal to the geometric genus for nonsingular curves \cite[Proposition IV.1.1]{Hartshorne}, the genus is independent of the arrangement of the pre-images and from \cite[Theorem 1.5]{FHIJMTZ} we get the simpler formula
        \begin{equation*}
            g = (N-3)2^{N-2}+1.
        \end{equation*}
    \end{proof}
    \begin{cor} \label{core_nonsing_a}
        If the curve (fiber) defining an arrangement of $2N$ rational pre-images of $a$ is nonsingular, then the genus is greater than $1$ for $2N \geq 8$.
    \end{cor}
    We have thus reduced the computation of $\bar{\kappa}(a,K)$ to checking $a$ values where the fiber is singular for arrangements with $8$ (or more) rational pre-images (224, 242, 2222).  The method is as follows.
    \begin{enumerate}
        \item Using the Jacobian criterion, determine all of the singular fibers ($a$ values).
        \item Determine the $\delta$-invariants of each singular point to determine the genus of each singular fiber.
    \end{enumerate}
    Recall that the $\delta$-invariant of a singularity $P$ is defined as
    \[ \delta_P = \sum_{Q} \frac{1}{2}m_Q(m_Q-1),\]
    where the sum ranges over the infinitely near points of $P$ and $m_Q$ are their multiplicities.  See \cite[Section 3.2]{RationalAlgebraicCurves} for the basic definitions and the case of plane curves and \cite[Section 9.2, Theorem 7]{PlaneAlgebraicCurves} for a more general discussion.
    As the singularity analysis computations are identical in form for all of the singularities, we outline the method, include the first such computation, and omit the details for the other singularities.  The singularity analysis proceeds as follows.
    \begin{enumerate}
        \item Let $C\subseteq \PP^N$ be a singular curve with singular point $P$. We move $P$ to $(0,\ldots,0,1)$ and dehomogenize.
        \item Project onto a singular plane curve with isomorphic tangent space at the singular point.
        \item Analyze the singularity of the plane curve with blow-ups and compute the $\delta$-invariant.
    \end{enumerate}

\subsection{Examining the 224 Surface}
    One possible 224 arrangement of $8$ pre-images is
    
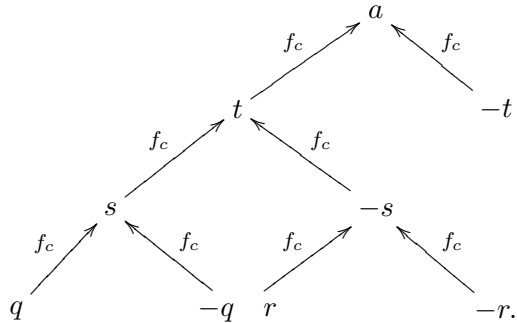
\begin{equation*}
        \xymatrix{& & & & a &\\
            & & & t \ar[ur]^{f_c}&  & -t\ar[ul]_{f_c} \\
            & & s \ar[ur]^{f_c} & & -s \ar[ul]_{f_c} &\\
            & q \ar[ur]^{f_c} & & -q \ar[ul]_{f_c} \quad r\ar[ur]^{f_c} & & -r.\ar[ul]_{f_c}}
    \end{equation*}
    Every other 224 arrangement differs only by renaming, so this is the only distinct 224 arrangement.  The curve is defined by three degree two equations in $\PP^4$ as
    \begin{equation*}
        C_{224} = V(az^2-t^2 - (tz-s^2),az^2-t^2 - (sz-q^2),az^2-t^2 - (-sz-r^2)) \subseteq \PP^4_{K(a)}.
    \end{equation*}
    \begin{thm} \label{thm_224}
        The $a$ values for which the fiber of the $224$ surface is singular are given by
        \begin{equation*}
            a \in \left\{-\frac{1}{4},0,a_1,a_2,a_3 \right\}
        \end{equation*}
        where $a_1,a_2,a_3$ are the three $3\rrd$ critical values of $f_c$
    \end{thm}
    \begin{proof}
    	We apply the Jacobian criterion to determine the singular points.  For each singular point, we can determine the associated $a$ value(s).
        Examining the hyperplane at infinity $z=0$ we have the $8$ cuspidal points $(\pm 1, \pm 1, \pm 1, 1, 0) \in \PP^4$.  To check the singularity of these points, we use the Jacobian criterion on the affine chart $\Aff^4_{q \neq 0}$ with generators
        \begin{equation*}
            \{az^2-t^2 - (tz-s^2), az^2-t^2 - (sz-1), az^2-t^2 - (-sz-r^2)\}
        \end{equation*}
        to have the Jacobian matrix at $z=0$
        \begin{equation*}
            \begin{pmatrix}
              0 & 2s & -2t & -t\\
              0 & 0 & -2t & -s\\
              2r & 0 & -2t & s
            \end{pmatrix}.
        \end{equation*}
        The determinant of one such maximal minor is $-8rst$, and since $r,s,t \neq 0$, this is nonzero, so the cuspidal points are all nonsingular.

        Now we consider the points in the affine chart $\Aff^4_{z \neq 0}$ which has generators
        \begin{equation*}
            \{a-t^2 - (t-s^2), a-t^2 - (s-q^2), a-t^2 - (-s-r^2)\}.
        \end{equation*}
        The Jacobian matrix is given by
        \begin{equation*}
            \begin{pmatrix}
              0 & 0 & 2s & -2t-1\\
              2q & 0 & -1 &  -2t\\
              0 & 2r & 1 &  -2t
            \end{pmatrix}
        \end{equation*}
        and the determinants of the maximal minors are
        \begin{equation*}
            \{8qrs, 4qr(-2t-1), 2q(4st-2t-1), -2r(4st +2t+1) \}.
        \end{equation*}
        The combinations that result in all $4$ determinants vanishing are the following.
        \begin{enumerate}
            \item If $q=r=0$, then we have $c=\pm s$ and so $c=0$ and so $a=0$.

            \item If $q=0$ and $(4st+2t+1)=0$, then we must have $s\neq -\frac{1}{2}$ so we can solve $t=-\frac{1}{4s+2} = -\frac{1}{4c+2}$.  Then we have $s^2+c =c^2+c=t$ and the roots of $4c^3 + 6c^2 + 2c + 1=\frac{df_c^3(0)}{dc}$ combined with $a = f_c(f_c(f_c(0)))$ to get the three $3\rrd$ critical values.

            \item If $q \neq 0$,  $r=0$, and $(4st - 2t- 1)=0$, then we must have $t\neq 0$ and we can solve $s=\frac{2t+1}{4t} = -c$.  Then we have $s^2-s = t$ and the roots of $16t^3 + 4t^2 - 1$ which give the three $3\rrd$ critical values.

            \item If $q,r \neq 0$, $s=0$, and $t=-\frac{1}{2}$, then we have $c=-\frac{1}{2}$ and so $a=-\frac{1}{4}$.
        \end{enumerate}
%
%
    \end{proof}
    We will treat $a=-\frac{1}{4}$ in Section \ref{sectminus1fourth}.
    \begin{thm} \label{thm_g224}
        The genus of $C_{224}$ is
        \begin{equation*}
            g = \begin{cases}
              4 & a=0\\
              1 & a \in \{a_1,a_2,a_3\}
            \end{cases}
        \end{equation*}
        where $a_1,a_2,a_3$ are the three $3\rrd$ critical values of $f_c$.
    \end{thm}
    \begin{proof}
        There is one singular point for $a=0$ and four singular points for each $a_i$.  In all cases $\delta_P = 1$ so the genus drops by 1 for each singular point.

        We now compute the $\delta$-invariant of one of the singular points for $a_1$.  The 224 curve for $a_1$ is defined as
        \begin{equation*}
            V(a_1z^2-t^2 - (tz-s^2),a_1z^2-t^2 - (sz-q^2),a_1z^2-t^2 - (-sz-r^2))
        \end{equation*}
        and if $\alpha$ is a root of
        \begin{equation*}
            4x^3+6x^2 + 2x + 1
        \end{equation*}
        then
        \begin{equation*}
            a_1 = \alpha^4 + 2\alpha^3 + \alpha^2 + \alpha = -1/4\alpha^2 + 1/2\alpha - 1/8.
        \end{equation*}
        We label the coordinates as $(q,r,s,t,z)$ and the singular point is
        \begin{equation*}
            P = (0,-\beta, \alpha, \alpha^2 + \alpha,1)
        \end{equation*}
        where
        \begin{equation*}
            \beta^2 = -2\alpha.
        \end{equation*}
        We move $P$ to $(0,0,0,0,1)$ with a translation
        \begin{equation*}
            (q,r,s,t,z) \mapsto (q,r-\beta z,s+\alpha z,t + (\alpha^2 + \alpha)z)
        \end{equation*}
        to get a new curve $\widetilde{C}$ and singular point $\widetilde{P} = (0,0,0,0,1)$.  We dehomogenize to affine coordinates $(Q,R,S,T) = \left(\frac{q}{z},\frac{r}{z},\frac{s}{z},\frac{t}{z}\right)$ and compute the tangent space at $\widetilde{P}$ as
    \begin{equation} \label{eq_ts}
        \left\{
        \begin{array}{l}
            -2T\alpha^2 -2T\alpha - T + 2S\alpha =0 \\
            -2T\alpha^2 - 2T\alpha - S =0 \\
            -2T\alpha^2 - 2T\alpha + S - 2\beta R=0.
        \end{array} \right.
    \end{equation}
        Notice that the second equation of (\ref{eq_ts}) implies the first using the degree 4 polynomial satisfied by $\alpha$.  Thus, the tangent space is given by
        \[
        	\left\{
        	\begin{array}{l}
            -2T\alpha^2 - 2T\alpha - S =0 \\
            -2T\alpha^2 - 2T\alpha + S - 2\beta R =0.
            \end{array} \right.
        \]
        Since we want to project $\widetilde{C}$ to a plane curve preserving the tangent space at $\widetilde{P}$ we define
        \begin{align*}
            u&=-2T\alpha^2 - 2T\alpha - S\\
            v&=-2T\alpha^2 - 2T\alpha + S - 2\beta R.
        \end{align*}
        with inverse
        \begin{align*}
            S &= \beta R -\frac{u}{2} + \frac{v}{2}\\
            T &= \frac{\beta R}{-2\alpha^2 - 2\alpha} + \frac{u}{-4\alpha^2 - 4\alpha} + \frac{v}{-4\alpha^2 - 4\alpha}
        \end{align*}
        and make the change of variables $(Q,R,S,T) \mapsto (Q,R,u,v)$ to get a new curve $\widetilde{C}'$ and point $\widetilde{P}'$.  The tangent space at $\widetilde{P}'$ is given by $u=v=0$.  We now project $\widetilde{C}'$ onto a plane curve in the $QR$-plane.
        To project we eliminate the variables $u,v$ from the three defining equations of $\widetilde{C}'$ to get the single equation
        \begin{align*}
            &(2\alpha + 1)Q^8 + ((-8\alpha - 4)R^2 + (16\beta\alpha + 8\beta)R + (16\alpha^2 - 4))Q^6 + ((12\alpha + 6)R^4 + (-48\beta\alpha - 24\beta)R^3 \\
            &+ (-144\alpha^2 - 64\alpha + 4)R^2 + (96\beta\alpha^2 + 32\beta\alpha - 8\beta)R + (-64\alpha^2 - 24\alpha - 8))Q^4 + ((-8\alpha - 4)R^6 \\
            &+ (48\beta\alpha + 24\beta)R^5 + (240\alpha^2 + 128\alpha + 4)R^4 + (-320\beta\alpha^2 - 192\beta\alpha - 16\beta)R^3 \\
            &+ (384\alpha^2 + 208\alpha + 128)R^2 + (-128\beta\alpha^2 - 96\beta\alpha - 64\beta)R - 32\alpha)Q^2 + ((2\alpha + 1)R^8 \\
            &+ (-16\beta\alpha - 8\beta)R^7 + (-112\alpha^2 - 64\alpha - 4)R^6 + (224\beta\alpha^2 + 160\beta\alpha + 24\beta)R^5 \\
            &+ (-320\alpha^2 - 152\alpha - 136)R^4 + (128\beta\alpha^2 + 32\beta\alpha + 96\beta)R^3 + (-64\alpha^2 + 64\alpha)R^2)=0
        \end{align*}
        defining a plane curve in $\Aff^2$ with variables $(Q,R)$.  Notice that the only points of the form $(0,0,u,v)$ on $\widetilde{C}'$ is the point $(0,0,0,0)$ (in the other words, the singular point is the only point that projects onto $(0,0)$), so we proceed with analyzing the plane curve singularity $(0,0)$.  Blowing-up once resolves the singularity and we see that it has multiplicity 2.  So we compute
        \begin{equation*}
            \delta_{P} = \frac{1}{2}(2 \cdot 1) = 1.
        \end{equation*}

        A similar analysis is done on all of the other singularities to get $\delta_P = 1$ for all $P$ for all $a \in \{0,a_1,a_2,a_3\}$.  Hence, we have
        \begin{equation*}
            \begin{cases}
              g = 5 - 1 = 4 & \text{if } a=0\\
              g = 5-(1+1+1+1) = 1 & \text{if }a=a_1,a_2,a_3.
            \end{cases}
        \end{equation*}

%
%

    \end{proof}

\subsection{Examining the 242 Surface}
    One possible 242 arrangement of $8$ pre-images is
    
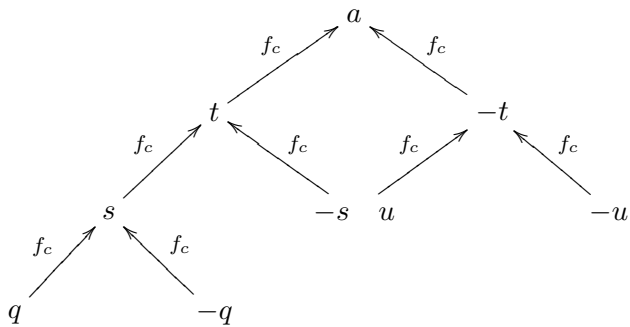
\begin{equation*}
        \xymatrix{& & & a  & &\\
            & & t \ar[ur]^{f_c}&  & -t\ar[ul]_{f_c}  &\\
            & s \ar[ur]^{f_c} & & -s \quad u \ar[ul]_{f_c} \ar[ur]^{f_c}& &-u. \ar[ul]_{f_c}\\
            q \ar[ur]^{f_c} & & -q \ar[ul]_{f_c}& & &}
    \end{equation*}
    Every other 242 arrangement differs only by renaming, so this is the only distinct 242 arrangement.  The surface is defined by $3$ degree two equations in $\PP^4$ as
    \begin{equation*}
        C_{242}=V(az^2-t^2 - (tz-s^2),az^2-t^2 - (-tz-u^2),az^2-t^2 - (sz-q^2)) \subseteq \PP^4_{K(a)}.
    \end{equation*}
    \begin{thm}
        The $a$ values for which the fiber of the 242 surface is singular are given by
        \begin{equation*}
            a \in \left\{-\frac{1}{4},0,2,a_1,a_2,a_3 \right\}
        \end{equation*}
        where $a_1,a_2,a_3$ are the three $3\rrd$ critical values of $f_c$.
    \end{thm}
    \begin{proof}
        We apply the Jacobian criterion to determine the singular points.  For each singular point, we can determine the associated $a$ value(s).  Examining the hyperplane at infinity, $z=0$, we have the $8$ cuspidal points $(\pm 1, \pm 1, \pm 1, 1, 0) \in \PP^4$.  To check the singularity of these points, we use the Jacobian criterion on the affine chart $\Aff^4_{q \neq 0}$ with generators
        \begin{equation*}
            \{az^2-t^2 - (tz-s^2), az^2-t^2 - (-tz-u^2), az^2-t^2 - (sz-1)\}.
        \end{equation*}
        The Jacobian matrix at $z=0$ is given by
        \begin{equation*}
            \begin{pmatrix}
              2s & 0 & -2t & - t\\
              0 & 2u & -2t &  t\\
              0 & 0 & -2t &  - s
            \end{pmatrix}.
        \end{equation*}
        The determinant of one maximal minor is $-8sut$, and since $s,u,t \neq 0$, this is nonzero, so the cuspidal points are all nonsingular.

        Now we consider the points in the affine chart $\Aff^4_{z \neq 0}$ which has generators
        \begin{equation*}
            \{a-t^2 - (t-s^2), a-t^2 - (-t-u^2), a-t^2 - (s-q^2)\}.
        \end{equation*}
        The Jacobian matrix is given by
        \begin{equation*}
            \begin{pmatrix}
              0 & 2s & -2t - 1 & 0\\
              0 & 0 & -2t +1 &  2u\\
              2q & -1 & -2t &  0
            \end{pmatrix}.
        \end{equation*}
        The determinants of the maximal minors are
        \begin{equation*}
            \{ 2u(4st +2t+1), 4qu(-2t-1), 8qus, 4qs(-2t+1) \}.
        \end{equation*}
        The combinations that result in all $4$ vanishing are as follows:
        \begin{enumerate}
            \item If $q=0$ and $u=0$, then $f_c^2(0) = a$ and $f_c^3(0) =a$ which is the polynomial equation
                \begin{equation*}
                    f_c(f_c(f_c(0)))-f_c(f_c(0)) = c^4 + 2c^3 = c^3(c+2) =0
                \end{equation*}
                so $c=0$ or $c=-2$.  So we have $a=0$ or $a=2$.

            \item If $q=0$ and $(4st+2t+1)=0$, then we must have $s\neq -\frac{1}{2}$ so we can solve $t=-\frac{1}{4s+2} = -\frac{1}{4c+2}$.  Then we have $s^2+c =c^2+c=t$ and the roots of $4c^3 + 6c^2 + 2c + 1=\frac{df_c^3(0)}{dc}$ combined with $a = f_c(f_c(f_c(0)))$ to get the three $3\rrd$ critical values.

            \item  If $u=0$ and $s=0$, then $c= \pm t$ and so $t=c=0$ and so $a=0$.

            \item If $u=0$ and $t=\frac{1}{2}$, then $c= -\frac{1}{2}$ and so $a=-\frac{1}{4}$.

            \item If $s=0$ and $t=-\frac{1}{2}$, then $c= -\frac{1}{2}$ and so $a=-\frac{1}{4}$.
        \end{enumerate}
    \end{proof}
    We will treat $a=-\frac{1}{4}$ in Section \ref{sectminus1fourth}.
    \begin{thm}
        The genus of $C_{242}$ is
        \begin{equation*}
            g = \begin{cases}
              3 & a=0\\
              4 & a=2\\
              3 & a \in \{a_1,a_2,a_3\}.
            \end{cases}
        \end{equation*}
    \end{thm}
    \begin{proof}
        We proceed as in the proof of Theorem \ref{thm_g224} for analyzing the singularities.

        For $a=0$ there is one singularity that required two blow-ups to resolve and we get multiplicity $2$ for both of the infinitely near points and, hence, $\delta_P= \frac{1}{2}\left(2 \cdot 1\right) + \frac{1}{2}\left(2 \cdot 1\right)=2$ and $g=5-2 = 3$.

        For $a=2$ there is one singular point with $\delta_P = 1$ and, hence, $g=5-1 =4$.

        For $a\in \{a_1,a_2,a_3\}$ each curve has two singular points both with $\delta_P = 1$ and, hence, $g=5-(1+1) = 3$.
%
%
%
%

\end{proof}

\subsection{Examining the 2222 Surface}
    One possible 2222 arrangement of $8$ pre-images is
    
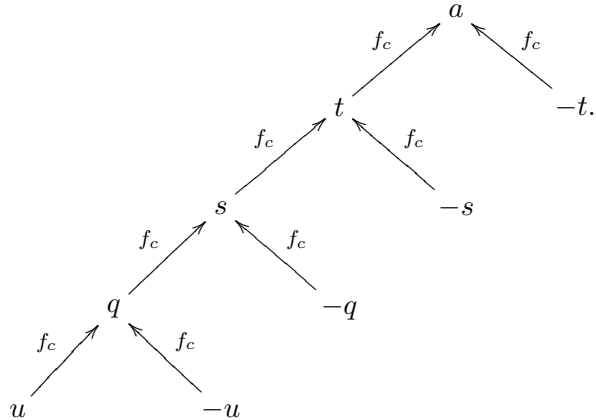
\begin{equation*}
        \xymatrix{&& & & a  &\\
            && & t \ar[ur]^{f_c}&  & -t. \ar[ul]_{f_c} &\\
            && s \ar[ur]^{f_c} & & -s \ar[ul]_{f_c} \\
            &q \ar[ur]^{f_c} & & -q \ar[ul]_{f_c}& \\
            u \ar[ur]^{f_c} & & -u\ar[ul]_{f_c}& }
    \end{equation*}
    Every other 2222 arrangement differs only by renaming, so this is the only distinct 2222 arrangement.  The surface is defined by $4$ degree two equations in $\PP^5$ as
    \begin{equation*}
        C_{2222}=V(az^2-t^2 - (tz-s^2),az^2-t^2 - (sz-q^2),az^2-t^2 - (qz-u^2)) \subseteq \PP^5_{K(a)}.
    \end{equation*}
    From \cite[Theorem 1.3]{FHIJMTZ} the only singular fibers are for $a$ the $N\tth$ critical values for $2 \leq N \leq 4$.  For $N=2$ we get $a=-\frac{1}{4}$ which will be treated in Section \ref{sectminus1fourth}.  For $N=3$ we get the three $3\rrd$ critical values which we label $a_{3,1},a_{3,2},a_{3,3}$.  For $N=4$ we get the seven $4\tth$ critical values which we label $a_{4,i}$ for $1 \leq i \leq 7$  which satisfy
    \begin{equation*}
        a=f_c(f_c(f_c(f_c(0)))) \quad \text{for }\quad 8c^7 + 28c^6 + 36c^5 + 30c^4 + 20c^3 + 6c^2 + 2c + 1=0.
    \end{equation*}
    \begin{thm}
        The genus of $C_{2222}$ is
        \begin{equation*}
            g = \begin{cases}
              3 & a \in \{a_{3,1}, a_{3,2}, a_{3,3}\}\\
              4 & a \in \{a_{4,i} \col 1 \leq i \leq 7\}.
            \end{cases}
        \end{equation*}
    \end{thm}
    \begin{proof}
        A fiber of the 2222 surface is isomorphic \cite[Proposition 4.2]{FH} to the degree $16$ plain curve defined by the equation
        \begin{equation*}
            f_c^4(x)=a.
        \end{equation*}
        For $a \in \{a_{3,i}\}$ there are three singular points, one of which is $(0,1,0)$ and the other two depend on $a$.  The $(0,1,0)$ point requires several blow-ups and has $\delta_P=100$ and each of the other two points have $\delta_P = 1$ for a final genus of
        \begin{equation*}
            g = \frac{1}{2}\left(15\cdot 14\right) - 102 = 105-102 = 3.
        \end{equation*}

        For $a \in \{a_{4,i}\}$ there are two singular points, one of which is $(0,1,0)$ and the other depends on $a$.  The $(0,1,0)$ point has $\delta_P=100$ and the point has $\delta_P = 1$ for a final genus of
        \begin{equation*}
            g = \frac{1}{2}\left(15\cdot 14\right) - 101 = 105-101 = 4.
        \end{equation*}.

    \end{proof}

    \begin{cor}\label{cor_rational_4th}
        For any $a \in \bar{\QQ} \backslash \left\{-\frac{1}{4}\right\}$ and any algebraic number field $K$ there are only finitely many $c \in K$ for which there are at least two $K$-rational $4\tth$ pre-images of $a$.
    \end{cor}

\subsection{The Bound $\bar{\kappa}(-1/4)$} \label{sectminus1fourth}
    For $a=-\frac{1}{4}$ the pre-images curves are in fact reducible since we have an equation in the generators of the form
    \begin{equation*}
        s^2 + (t-1/2z)^2 = (s-(t-1/2z))(s+(t-1/2z)),
    \end{equation*}
    where $s$ is a $2\nnd$ pre-image of $a$ for which $s^2+c= t$ and $t^2+c = a$ and an equation of the form
    \begin{equation*}
        u^2 -(t+1/2z)^2 = (u-(t+1/2z)(u-(t+1/2z))
    \end{equation*}
    where $u$ is a $2\nnd$ pre-image of $a$ for which $u^2+c = -t$.  After splitting the pre-image curves into their distinct irreducible components we can again proceed with genus calculations.
    \begin{thm}\label{thm_minus14}
        For any fixed number field $K$, $\bar{\kappa}\left(-\frac{1}{4}\right) = 10$.
    \end{thm}
    \begin{proof}
        Using the Jacobian criterion we compute that the following curves are all nonsingular, and we apply the genus formula from \cite[\S 22]{Hirzebruch} or \cite[Corollary 2]{Arslan} to compute the following genera.
        \begin{equation*}
            g = \begin{cases}
              1 & \{224,2222,244,2422\}\\
              5 & \{22222,2224,2242,246,2442,2424,24222\}.
            \end{cases}
        \end{equation*}
        Using Magma, we see that the $244$ curve is a rank $1$ elliptic curve over $\QQ$ isomorphic to
        \begin{equation*}
            v^2w = u^3 + u^2w - 9uw^2 + 7w^3
        \end{equation*}
        so has infinitely many rational points.  Therefore, there are infinitely many $c$ with $10$ rational pre-images of $-\frac{1}{4}$ and only finitely many $c$ values with $12$ (or more) rational pre-images of $-\frac{1}{4}$.
    \end{proof}

\section{Proof of Theorem \ref{thm_kappabar}} \label{sect_proof}

    We are now ready to prove Theorem \ref{thm_kappabar}.
    \begin{proof}
        The case $a = -\frac{1}{4}$ was covered in Theorem \ref{thm_minus14}.

        For $a$ a third critical value we have genus $1$ for the 224 curve and, hence, for a large enough extension of $\QQ$ it has positive rank and infinitely many rational points.  Also, it has no $\QQ$-rational points.  The 242 curve has genus greater than $1$ and, hence, has only finitely many rational points.  Thus, for $\bar{\kappa}(a,K)$ to be at least $10$ there must be infinitely many rational points on a curve corresponding to an arrangement with rational $4\tth$ pre-images, which is not possible by Corollary \ref{cor_rational_4th}.  So it is possible for $\bar{\kappa}(a,K)$ to be either $6$ or $8$ depending on the field.

        For all other values of $a$ we have the genus of the 224 and 242 curves are greater than $1$ and, hence, have only finitely many rational points.  Any arrangement with more points must contain one of these two arrangements, hence $\bar{\kappa}(a,K) \leq 6$.  Theorem \ref{thm222_rank} shows that the $222$ surface has generic rank $2$ and \cite{Zannier} shows that the set of $a$ where the rank is $0$ is finite.  Every $a$ value for which both $E_{222}$ and $E_{24}$ specialize to rank $0$ has $\bar{\kappa}(a) = 4$, otherwise $\bar{\kappa}(a) = 6$.
    \end{proof}

\section{Other Properties of Pre-Image Surfaces}\label{sect_other}
    In this section we collect some additional properties of the pre-images surfaces that are tangential to the proof of Theorem \ref{thm_kappabar}, yet still of interest.

\subsection{Parametrization of Torsion Subgroups of $E_{24}$}\label{sect_torsion}
        Recall that Mazur's Theorem \cite{Mazur} gives a description of the possible torsion subgroups of elliptic curves over $\QQ$ and that the specialization map is injective on nonsingular fibers.  These facts combined with Theorem \ref{thm24_rank} implies that the possible torsion subgroups for a nonsingular specialization of $E_{24}(a)$ must be isomorphic to one of the following groups:
        \begin{equation*}
            \{\ZZ/2\ZZ\times\ZZ/4\ZZ,\;\ZZ/2\ZZ\times\ZZ/8\ZZ,\;\ZZ/4\ZZ,\;\ZZ/8\ZZ,\;\ZZ/12\ZZ\}.
        \end{equation*}
        We characterize the $a$ values giving rise to a specialization with each of these possible torsion subgroups in the following theorem.
        \begin{thm} \label{thm24_torsion}
            \mbox{}
            \begin{enumerate}
                \item $E_{24}(a)(\QQ)$ contains a subgroup isomorphic to $\ZZ/2\ZZ\times\ZZ/4\ZZ$ if and only if
                    \begin{equation*}
                        a=-t^2 \quad \text{for} \quad t\in\QQ \backslash \left\{0,\pm\frac{1}{2}\right\}.
                    \end{equation*}

                \item $E_{24}(a)(\QQ)$ contains a subgroup isomorphic to $\ZZ/8\ZZ$ if and only if
                    \begin{equation*}
                        a=\frac{t^2(t^2 - 2)}{4} \quad \text{for} \quad t\in\QQ\backslash \{0,\pm1\}.
                    \end{equation*}

                \item $E_{24}(a)(\QQ)$ contains a subgroup isomorphic to $\ZZ/2\ZZ \times \ZZ/8\ZZ$ if and only if
                    \begin{equation*}
                        a = -\frac{(4t^2 - 4t - 1)^2(4t^2 + 4t - 1)^2}{4(4t^2 + 1)^4} \quad \text{for} \quad t \in \QQ\backslash \left\{0,\pm \frac{1}{2}\right\}.
                    \end{equation*}

                \item $E_{24}(a)(\QQ)$ contains a subgroup isomorphic to $\ZZ/12\ZZ$ if and only if
                \begin{align*}
                    a &= \frac{(13691470144t^2 - 235376t + 1)(13903463744t^2 - 235376t + 1)^3}{9527265101250297856000000t^6(117688t-1)^2} \quad \text{for} \quad t \in \QQ \backslash \left\{0,\frac{1}{117688}\right\}.
                \end{align*}

%
            \end{enumerate}
        \end{thm}
        \begin{proof}
            \mbox{}
            \begin{enumerate}
                \item First suppose $a=-t^2$ for some $t\in\QQ \backslash \left\{0,\pm \frac{1}{2}\right\}$. Then $$\{ \OO, (4t^2+1,0,1), (4t,0,1),(-4t,0,1)\}$$
                    is a subgroup of $E_{24}(-t^2)(\QQ)$ isomorphic to $\ZZ/2\ZZ\times\ZZ/2\ZZ$.  Since there is also a generic torsion point of order $4$ (Theorem \ref{thm24_rank}), $E_{24}(-t^2)(\QQ)$ contains a subgroup isomorphic to $\ZZ/2\ZZ\times\ZZ/4\ZZ$.  Next, suppose $E_{24}(a)(\QQ)$ contains a subgroup isomorphic to $\ZZ/2\ZZ\times\ZZ/2\ZZ$ and, hence, also a subgroup isomorphic to $\ZZ/2\ZZ\times\ZZ/4\ZZ$. Thus, $E_{24}(a)(\QQ)$ has three points of order two. Points of order two must be rational roots of the Weierstrass equation
                    \begin{equation} \label{eq_24}
                        x^3 +(4a-1)x^2+(16a)x+16a(4a-1) = (x+4a-1)(x^2+16a).
                    \end{equation}
                    So, $x^2+16a$ must have 2 rational roots, or equivalently, $a=-(x/4)^2=-t^2$. Hence, there are three rational roots of (\ref{eq_24}) if and only if $a=-t^2$ for $t\in\QQ$. However, if $t=\pm1/2$ then the roots will not be distinct, so we must have $a=-t^2$ for $t\in\QQ\backslash \{\pm \frac{1}{2}\}$.  For $t=0$ we get $a=0$ which is a degenerate case (a singular fiber of $\cpc{2}$).

                \item Suppose $a=t^2(t^2 - 2)/4$ for some $t\in\QQ \backslash\{0,\pm1\}$. Then it can be verified directly that the point $P=(2t(t^2+t-1),2(t-1)t(t+1)^3,1)$ is in $E_{24}(a)(\QQ)$ and $[2]P = (2,2(4a+1),1)$ is the generator of the cyclic subgroup of order four. So, $P$ generates a cyclic group of order eight.

                    Now suppose that $E_{24}(a)(\QQ)$ has a cyclic subgroup of order eight. If we let $P=(x,y,1)$ be the generator of the subgroup, then $[2]P$ generates a cyclic group of order four (the generic torsion subgroup). So, we must have $x([2]P)=2$. This gives us the equation
                    \begin{equation*}
                        x^4-8x^3-64ax^2+8x^2-512a^2x-1024a^3+256a^2+64a=0.
                    \end{equation*}
                    Then using the solution to the quartic we have the solutions
                    \begin{align*}
                        x&=2\pm2\sqrt{4a+1}+\frac{1}{2}\sqrt{24+(8a-1)\pm\frac{512+4096a^2+256(8a-1)}{16\sqrt{4a+1}}}\\
                        x&=2\pm2\sqrt{4a+1}-\frac{1}{2}\sqrt{24+(8a-1)\pm\frac{512+4096a^2+256(8a-1)}{16\sqrt{4a+1}}}.
                    \end{align*}
                    In order to have $x\in\QQ$, and since $x$ is clearly not $2$, we must have $\sqrt{4a+1}\in\QQ$. So $a=\frac{b^2-1}{4}$ for some $b\in \QQ$. The above roots become
                    \begin{align*}
                        x&=2(1\pm b + b\sqrt{1\pm b})\\
                        x&=2(1\pm b - b\sqrt{1\pm b})
                    \end{align*}
                    from which it follows that $b=\pm (t^2 - 1)$. Thus, $a=\frac{t^2(t^2-2)}{4}$. Note that for $t=\pm 1$ we get $a=-\frac{1}{4}$ and for $t=0$ we get $a=0$ which are all singular fibers.

                \item Clearly, $E_{24}(a)(\QQ)$ has a subgroup isomorphic to $\ZZ/2\ZZ \times \ZZ/8\ZZ$ if and only if $E_2(a)$ has a subgroup isomorphic to $\ZZ/2\ZZ \times \ZZ/4\ZZ$ and a subgroup isomorphic to $\ZZ/8\ZZ$. From the two previous parts, it follows that $a=-t_1^2$ and $a=\frac{t_2^2(t_2^2-2)}{4}$. These two equations define a curve of genus zero which can be parameterized with Magma and substituted into $a=-t_1^2$ to get the stated form.  For $t=0, \pm \frac{1}{2}$ we get $a=-\frac{1}{4}$, which is a singular fiber.

                \item Since specialization is injective on torsion for nonsingular fibers , $E_{24}(a)(\QQ)$ has a subgroup isomorphic to $\ZZ/12\ZZ$ if and only if there is a point $Q=[x,y] \in E_{24}(a)(\QQ)$ for which $[3]Q$ generates the generic $\ZZ/4\ZZ$ torsion subgroup.  In particular, we must have $x([3]Q) = 2$.  So we need to find solutions to
                    \begin{equation*}
                        \frac{x([3]Q)-2}{x-2} = 0
                    \end{equation*}
                    where we divide out by $x-2$ since we only wish to exclude the $a$ values which have purely $\ZZ/4\ZZ$ torsion.  From the \textit{algcurve} package in Maple we get the parametrization given.  The two excluded $t$ values correspond to the two singular fibers $a=0$ and $a=-\frac{1}{4}$.
            \end{enumerate}
        \end{proof}

        \begin{cor}
            The $a \in \QQ$ for which $E_{24}(a)(\QQ)$ has torsion subgroup exactly $\ZZ/4\ZZ$, in other words, the $a \in \QQ$ for which the specialization map is an isomorphism on torsion, is a Zariski dense set.
        \end{cor}
        \begin{proof}
            From Mazur's theorem and the injectivity of the specialization map, the possible torsion groups of $E_{24}(a)(\QQ)$ are
            \[
            \{\ZZ/2\ZZ\times\ZZ/4\ZZ,\;\ZZ/2\ZZ\times\ZZ/8\ZZ,\;\ZZ/4\ZZ,\;\ZZ/8\ZZ,\;\ZZ/12\ZZ\}.
            \]
            The condition on $a$ for $E_{24}(a)(\QQ)_{\tors}$ to not be $\ZZ/4\ZZ$ is a closed condition from Theorem \ref{thm24_torsion} and the $j$-invariant.  Therefore, every $a \in \QQ$ outside of this Zariski closed set satisfies $E_{24}(a)(\QQ)_{\tors} \cong \ZZ/4\ZZ$ and there is at least one such $a$, $E_{23}(1)(\QQ)_{\tors} \cong \ZZ/4\ZZ$.
        \end{proof}

\subsection{Exceptional $(c,a)$ Values Over $\QQ$} \label{sect_exceptional}
    \subsubsection{Rank Zero}
        The methods of \cite{Masser, Zannier}, in principle, can compute the full set $S$, but in practice such computations are difficult.  However, computing the set $S \cap K$ for $[K\colon \QQ] \leq 2$ from Theorem \ref{thm_kappabar} is feasible since we have an explicit (small) bound on the order of a torsion point.

        We must have both $P(a)$ and $Q(a)$ are torsion on the 222 surface.  We have a bound of $18$ for the order of a torsion point over a quadratic number field $K$ \cite{Kamienny,Kenku}.  Finding the $a$ for which $P(a)$ or $Q(a)$ is torsion of a given order is solving polynomials equation in $a$.  If there are any $a$ values for which they are both torsion, we compute the rank of $E_{24}(a)$.
        \begin{thm}
            Let $S$ be the set of $a$ values from Theorem \ref{thm_kappabar} for which $\bar{\kappa}(a) = 4$.  Let $K$ be a quadratic number field.  Then, $S \cap K= \emptyset$.
        \end{thm}
        \begin{proof}
            Direction computation.
        \end{proof}

%
%
%

    \subsubsection{Full Trees of Pre-Images}
        We can find an $a$ value with arbitrarily many $\QQ$-rational pre-images by taking $a$ to be the $n\tth$ forward image of any wandering $\QQ$-rational point.  This gives a very deep but potentially sparse pre-image tree.  Consequently, one may ask if you can find an $a$ and $c$ which gives a full tree to some level.  Clearly, if you allow $K/\QQ$ to be of large degree, the answer is any level, so we address this question over $\QQ$.  For example, the following lists $(c,a)$ with a $246$ pre-image arrangement.
        \begin{enumerate}
            \item $(-5248/2025, 726745984/284765625)$
            \item $(-17536/5625, 878382976/244140625)$
            \item $(-9153/6400, -437896611/400000000)$
            \item $(-24361/14400, -42/25)$
            \item $(-20817/25600, -1078371711/6400000000)$
            \item $(-180625/97344, 2845625/5483712)$
            \item $(-158848/99225, 20844352384/683722265625)$
        \end{enumerate}
        \begin{rem}
            We were unable to find any pairs $(c,a)$ over $\QQ$ with the full 248 arrangement, but it seems reasonable to expect that such an arrangement exists. We searched by choosing the smallest $3\rrd$ pre-image having height at most $\log(30,000)$, since choosing two third pre-images which map to same second pre-image (up to sign) fixes a unique $c$ value and, hence, a unique $a$ value.
        \end{rem}

%

\end{document}